\documentclass[12pt]{article}
\usepackage{color}
\usepackage{bm}
\advance \textwidth -60truemm\relax

\advance\oddsidemargin -1truein\relax

\evensidemargin=\oddsidemargin

\advance\textheight -60truemm\relax
\advance\topmargin -1truein\relax
\unitlength\textwidth
\divide\unitlength by 150\relax

\usepackage{amsmath,amssymb}
\usepackage{bm}

\bmdefine{\aaa}{a}
\bmdefine{\bbb}{b}
\bmdefine{\ccc}{c}
\bmdefine{\eee}{e}
\bmdefine{\xxx}{x}
\bmdefine{\yyy}{y}
\bmdefine{\zzz}{z}
\bmdefine{\zerovec}{0}

\newcommand{\CCC}{\mathbb{C}}
\newcommand{\RRR}{\mathbb{R}}
\newcommand{\FFF}{\mathbb{F}}

\newcommand{\rank}{\mathrm{rank}}

\makeatletter
\def\proof{\@ifnextchar[{\@proof}{\@proof[]}}
\def\@proof[#1]{\ifx#1\relax\relax{\noindent\bf Proof}\else{\noindent\bf Proof #1}\fi\quad}

\makeatother

\numberwithin{equation}{section}
\newtheorem{thm}[equation]{Theorem}
\newtheorem{lemma}[equation]{Lemma}
\newtheorem{prop}[equation]{Proposition}

\begin{document}
\title{Perfect type of $n$-tensors}
\author{Toshio Sumi\footnote{Kyushu University, Faculty of Design, 
4-9-1 Shiobaru, Minami-ku, Fukuoka, 815-8540, JAPAN,
e-mail: sumi@design.kyushu-u.ac.jp},
Toshio Sakata\footnote{Kyushu University, Faculty of Design, 
4-9-1 Shiobaru, Minami-ku, Fukuoka, 815-8540, JAPAN,
e-mail: sakata@design.kyushu-u.ac.jp},
and Mitsuhiro Miyazaki\footnote{Kyoto University of Education, 
Department of Mathematics,
1 Fujinomoricho, Fukakusa, Fushimi-ku, Kyoto, 612-8522, JAPAN,
e-mail: g53448@kyokyo-u.ac.jp}}
\maketitle
\begin{abstract}
In various application fields, tensor type data are used 
recently and then a typical rank is important.  
Although there may be more than one typical ranks over the real number field,
a generic rank over the complex number field is the minimum
number of them.
The set of $n$-tensors of type
$p_1\times p_2\times\cdots\times p_n$ is called perfect, if
it has a typical rank $\max(p_1,\ldots,p_n)$.
In this paper, we determine perfect types of $n$-tensor.

%\begin{keywords} tensor; typical rank; Jacobian \end{keywords}
%\begin{classcode}15A69, 15A72, 14Q99, 14M12, 14M99\end{classcode}

\end{abstract}
%%%%%%%%%%
%\begin{document}
\section{Introduction}

An $p_1\times p_2\times\cdots\times p_n$ tensor over a field $\FFF$ is an element of 
the tensor product of $n$ vector spaces 
$\FFF^{p_1}, \FFF^{p_2},\ldots, \FFF^{p_n}$.
Thus every tensor can be expressed as a sum of tensors of the form
$\aaa_1\otimes \aaa_2\otimes\cdots\otimes \aaa_n$
for $\aaa_i \in \FFF^{p_i}$, $i=1,2,\ldots,n$. 
The rank $\rank_{\FFF}\, T$ of a tensor $T$ means that the minimum number $r$ 
of rank one tensors which express $T$ as a sum.
The rank depends on the field. 

The set $T(p_1,\ldots,p_n\/;\FFF)$ of all $p_1\times\cdots\times p_n$ tensors
is $\FFF^{p_1}\times\cdots\times \FFF^{p_n}$ as a set.
We consider the Euclidean topology on 
$\FFF^{p_1}\times\cdots\times \FFF^{p_n}=\FFF^{p_1\cdots p_n}$ 
as a topology on the set $T(p_1,\ldots,p_n\/;\FFF)$.

Now let $\FFF$ be the real number field $\RRR$ or the complex number 
field $\CCC$.
A typical rank, denoted by ${\rm typical\_rank}_{\FFF}(p_1,\ldots,p_n)$, 
of $T(p_1,\ldots,p_n\/;\FFF)$ is defined as the set of integers $r$ such that
the set of rank $r$ tensors 
has a positive Lebesgue measure in $T(p_1,\ldots,p_n\/;\FFF)$.
A typical rank of tensors is one of important tools for experimental
simulation. 
We know a typical rank of $3$-tensors of special types.
%%(cf. \cite{Miwa:2004a,Vasilescu-Terzopoulos:2005,Muti-Bourennane:2007}).
ten Berge obtained that
the typical rank of $m\times n\times 2$ tensors is
$\min(n,2m)$ if $2\leq m< n$ and
$\{\min(n,2m),\min(n+1,2m)\}$ if $2\leq m=n$ \cite{tenBerge-etal:1999}, and
the minimum number of the typical rank of $m\times n\times p$ tensors 
with $3\leq m\leq n$ is just $\min(p,mn)$ if $p\geq (m-1)n$ \cite{tenBerge:2000}
over the real number field.
In \cite{Miyazaki-etal:2009} we considered a generic form of 
$m\times n\times 3$ tensors.
Recently, Comon et al. \cite{Comon-etal:2009} studied the minimum number of 
the typical rank of $3$-tensors by using the 
Jacobian of the map
$$\{\aaa(r),\bbb(r),\ccc(r)\} \to T=\sum_{r=1}^R \aaa(r)\odot \bbb(r)\odot \ccc(r).$$
In contrast to 
that there may be more than one typical ranks over the real number field,
we remark that a typical rank of $n$-tensors 
over the complex number field consists
of just one number and thus it is called a generic rank.
In this paper,
we consider the smallest typical rank of $n$-tensors
over the real number field. 
It is equal to the unique typical rank of $n$-tensors
over the complex number field (cf. \cite{Northcott:1980}).

A format $(p_1,\ldots,p_n)$ is called \lq\lq{}perfect\rq\rq{} if
$\max(p_1,\ldots,p_n)$ is a typical rank of $T(p_1,\ldots,p_n;\RRR)$.
%%%%%
Suppose that $2\leq p_1\leq p_2\leq p_3$.
In \cite{tenBerge:2000}, $p_1\times p_2\times p_3$ tensor is called 
\lq\lq{}tall\rq\rq{} if $p_1p_2-p_2<p_3<p_1p_2$ 
and tall $p_1\times p_2\times p_3$ tensors have a unique typical rank $p_3$.
Thus $(p_1,p_2,p_3)$ is perfect if $p_1p_2-p_2<p_3\leq p_1p_2$.
More generally, 
if $p_1p_2-p_1-p_2+2\leq p_3\leq p_1p_2$ then $(p_1,p_2,p_3)$ is perfect 
(see \cite[exercise 20.6, page 535]{Buergisser-etal:1997}).
We extend this result for $n$-tensors.
Our main theorem is as follows.

\begin{thm} \label{thm:main}
Suppose that $n\geq 2$ and $2\leq p_1\leq \cdots\leq p_{n}$.
Let $q=p_1\cdots p_{n}-(p_1+\cdots +p_{n})+n$.
If $q\leq p_{n+1}\leq p_1\cdots p_{n}$
then $p_{n+1}$ is the smallest typical rank of 
$p_1\times \cdots\times p_{n+1}$ tensors and $(p_1,\ldots,p_{n+1})$ is perfect.
Conversely if $(p_1,\ldots,p_{n+1})$ is perfect
then $q\leq p_{n+1}\leq p_1\cdots p_{n}$.
\end{thm}

We show the theorem in the next section.

%%%%%%%%%%%%%%%%%%%%%%%%%%%%%%%%%%%%%%%
\section{Proof of Theorem~\ref{thm:main}}
%%%%%%%%%%%%%%%%%%%%%%%%%%%%%%%%%%%%%%%

In this section we give a proof of Theorem~\ref{thm:main}.
First we give a range of typical ranks.

\begin{lemma} \label{lem:minimal typical rank}
Let $2\leq p_1\leq p_2\leq \cdots \leq p_{n+1}\leq p_1\cdots p_{n}$.
A typical rank of $p_1\times\cdots\times p_{n+1}$ tensors is
greater than or equal to $p_{n+1}$ and less than or equal to $p_1p_2\cdots p_{n}$.
\end{lemma}

\begin{proof}
Let $A=(A_1;\cdots;A_{p_{n+1}})$ be an $p_1\times\cdots\times p_{n+1}$ tensor,
where $A_j$ is a $p_1\times\cdots\times p_{n}$ tensor
for $j=1,\ldots,p_{n+1}$.
Let consider the vector space $V$ spanned by $A_1,\ldots,A_{p_{n+1}}$.
We denote by $f(A_j)$ a column vector given by flattening of $A_j$.
Note that 
$$\rank(A)\geq \rank(f(A_1),\ldots,f(A_{p_{n+1}}))=\dim V.$$
If $\dim V<p_{n+1}$ then all $p_{n+1}$-minors of the matrix 
$\begin{pmatrix} f(A_1)&\cdots,f(A_{p_{n+1}})\end{pmatrix}$ are zero.
Thus $\{(X_1;\cdots;X_{p_{n+1}}) \mid \dim \langle X_1,\ldots,X_{p_{n+1}}\rangle=p_{n+1}\}$ 
is a Zariski open set in $T(p_1,\ldots,p_{n+1}) \cong \FFF^{p_1\cdots p_{n+1}}$.
Thus a typical rank is greater than or equal to $p_{n+1}$.
\par
In general $A=(a_{i_1i_2\ldots i_ni_{n+1}})$ is described as a sum of $p_1\cdots p_n$ rank one tensors
$$\eee_{i_1}^{(1)}\odot\cdots\odot\eee_{i_{n}}^{(n)}\odot 
  (a_{i_1\ldots i_{n}1},
\ldots, a_{i_1\ldots i_{n}p_{n+1}}),$$
where $e_{i}^{(j)}$ is the $i$-th row vector of the $p_j\times p_j$
identity matrix.
Thus $\rank(A)\leq p_1\cdots p_n$.
\end{proof}

Let $\varphi_1\colon \RRR^{p_1+\cdots+p_n}\to T(p_1,\ldots,p_n)$ be a map
defined by 
$$\varphi_1(\aaa_{1},\ldots,\aaa_{n})
=\aaa_{1}\odot \cdots\odot \aaa_{n}$$
and
$\varphi\colon \RRR^{(p_1+\cdots+p_n)r}\to T(p_1,\ldots,p_n)$ be a map
defined by 
$$\varphi(\aaa^{(1)}_{1},\ldots,\aaa^{(1)}_{n},\ldots,
  \aaa^{(r)}_{1},\ldots,\aaa^{(r)}_{n})
=\sum_{h=1}^r \varphi_1(\aaa^{(h)}_{1},\ldots,\aaa^{(h)}_{n}).$$
Put 
\begin{equation}\label{eq:phi1}
\phi_1(\aaa_1,\ldots,\aaa_n):=
\begin{pmatrix} E_{p_1}\otimes \aaa_2\otimes\cdots\otimes \aaa_n \\
  \aaa_1 \otimes E_{p_2} \otimes\cdots\otimes \aaa_n \\
  \vdots \\
  \aaa_1 \otimes\cdots\otimes \aaa_{p_{n-1}} \otimes E_{p_n} \end{pmatrix}
\end{equation}
for $\aaa_1\in\RRR^{p_1}$, $\ldots$, $\aaa_n\in\RRR^{p_n}$.
Then the Jacobian $J(\varphi)$ of $\varphi$ at 
$$(\aaa^{(1)}_{1},\ldots,\aaa^{(1)}_{n},\ldots,
  \aaa^{(r)}_{1},\ldots,\aaa^{(r)}_{n})$$
is given by 
$$\begin{pmatrix} \phi_1(\aaa^{(1)}_{1},\ldots,\aaa^{(1)}_{n}) \\ \vdots \\ 
  \phi_1(\aaa^{(r)}_{1},\ldots,\aaa^{(r)}_{n}) \end{pmatrix}.
$$
%%(cf. \cite{Comon-etal:2009}).
\par
If $r$ is a typical rank of $T(p_1,p_2,p_3)$ then
$$\frac{p_1p_2p_3}{p_1+p_2+p_3-2}\leq r\leq \min(p_1p_2,p_1p_3,p_2p_3)$$
\cite{Howell:1978,Buergisser-etal:1997}.
%% (\cite{Howell:1978}, \cite[Proposition 20.4 (5)]{Buergisser-etal:1997}).
This result also holds for $n$-tensors.

\begin{prop}%%[{\cite[Proposition 20.4]{Buergisser-etal:1997}}]
\label{prop:typical_rank_range}
A typical rank of $p_1\times\cdots\times p_n$ tensors
is greater than or equal to
$$\frac{p_1p_2\cdots p_n}{p_1+p_2+\cdots+p_n-n+1}$$
and less than or equal to 
$$\min(p_2p_3\cdots p_n,p_1p_3\cdots p_n,
\ldots, p_1p_2\cdots p_{n-1}).$$
\end{prop}

\begin{proof}
Let consider the Segre embedding which is a map of projective spaces
$$RP^{p_1-1}\times\cdots\times RP^{p_n-1} \to RP^{p_1\cdots p_n-1}$$
induced by the tensor product map $\varphi_1$.
The image $\textrm{im}(\varphi_1)$ has dimension $p_1+p_2+\cdots+p_n-n$.
Since $\{\aaa_1\odot\ldots\odot \aaa_n \mid \aaa_j\in\RRR^{p_j}\}$ is the affine cone of $\textrm{im}(\varphi_1)$,
it's dimension is $p_1+p_2+\cdots+p_n-n+1$.
If $r$ is a typical rank of $T(p_1,\ldots,p_n)$,
then $\dim T(p_1,\ldots,p_n) \leq r\dim(\textrm{im}(\varphi_1))$
and thus 
$$r\geq \frac{p_1\cdots p_n}{p_1+p_2+\cdots+p_n-n+1}.$$
\end{proof}

%%%%%%%%%%%%%%%%%%%%%%%%%%%%%%%%%%%%%%%%%%%%
From now on, 
let $2\leq p_1\leq p_2\leq \cdots \leq p_n$
and put $q=p_1p_2\cdots p_n-(p_1+p_2+\cdots +p_n) +n$.
Suppose that $q\leq p_{n+1}\leq p_1p_2\cdots p_n$.
By Lemma~\ref{lem:minimal typical rank} it suffices to show
that the Jacobian $J(\varphi)$ has full rank at some point.
\par
Let $S$ be a subset of 
$$\{(k_1,\ldots,k_n) \mid 1\leq k_j\leq p_j,\ j=1,\ldots n\}$$
with cardinality $p_{n+1}$ which contains 
$$S_0=\{(k_1,\ldots,k_n) \mid 
  1\leq k_j\leq p_j,\ \#\{j\mid k_j=p_j\}\ne n-1\}$$
and let $f\colon S \to \{1,2,\ldots, p_{n+1}\}$ be a bijection.

We define maps $u_1,u_2,\ldots,u_n$ by
$u_j(x_1,\ldots,x_n)=0$ if $x_j=p_j$, 
$u_j(x_1,\ldots,x_n)=1$ if $x_s=p_s$ for some $s\ne j$ and
otherwise
%%$u_j(i'_1,\ldots,i'_n)={i'_1}^j+{i'_2}^j+{i'_3}^j+1$,
$u_j(x_1,\ldots,x_n)={x_j}+1$,
for $j=1,\ldots,n$.

We denote by $\eee_j$ the $j$th row vector of the identity matrix.
We put $\aaa^{(h)}_{k} \in \RRR^{p_h}$, $h=1,\ldots, n+1$,
as
\begin{eqnarray*}
\aaa^{(h)}_{f(k_1,\ldots,k_n)} &=&\eee_{k_h}+u_h(k_1,\ldots,k_n)\eee_{p_h}, 
  \quad 1\leq h \leq n\\
\aaa^{(n+1)}_{f(k_1,\ldots,k_n)} &=& \eee_{f(k_1,\ldots,k_n)} 
\end{eqnarray*}
for all $(k_1,\ldots,k_n)\in S$.
%%%%%

We denote the row vector $\xxx$ as $(x(k_1,\ldots,k_{n+1}))$ if
$$\xxx=\sum_{k_1,\ldots,k_{n+1}} x(k_1,\ldots,k_{n+1})
  \eee_{k_1}\otimes \cdots \otimes \eee_{k_{n+1}}.$$ 
Let $g\colon \RRR^{p_1\cdots p_{n+1}} \to 
  \RRR[x(1,\ldots,1),\ldots, x(p_1,\ldots,p_{n+1})]$
be a map defined by
$$g(\sum_{k_1,\ldots,k_{n+1}} h_{k_1,\ldots,k_{n+1}}
  \eee_{k_1}\otimes \cdots \otimes \eee_{k_{n+1}})
  =\sum_{k_1,\ldots,k_{n+1}} h_{k_1,\ldots,k_{n+1}} 
     x(k_1,\ldots,k_{n+1}).$$
Note that $g$ is linear, that is, it holds that
$$g(s_1\yyy_1+s_2\yyy_2)=s_1g(\yyy_1)+s_2g(\yyy_2)$$
for $s_1,s_2\in \RRR$ and $\yyy_1,\yyy_2 \in \RRR^{p_1\cdots p_{n+1}}$.
%%%%%
We abbreviate $\eee_{i_1}\otimes\cdots\otimes \eee_{i_n}$ to 
$\eee(i_1,\ldots,i_n)$,
$u_j(k_1,\ldots,k_n)$ to $u_j$, and 
$u_j(i'_1,\ldots,i'_n)$ to $v_j$.
Then $x(i_1,\ldots,i_n)=g(\eee(i_1,\ldots,i_n))$.

Put
$$\zzz=(\aaa^{(1)}_1,\ldots,\aaa^{(n+1)}_1,\ldots,
\aaa^{(1)}_{p_{n+1}},\ldots,\aaa^{(n+1)}_{p_{n+1}}).
$$
We prepare three lemmas to show that
the equation $J(\varphi(\zzz))\xxx^T=\zerovec$ 
has no nonzero solution.

\begin{lemma} \label{lem:codim1}
Let $n\geq 2$.
Suppose that
$$g((\eee_{k_1}+\eee_{p_1}) \otimes \cdots \otimes (\eee_{k_{n}}+\eee_{p_{n}})) = 0$$
for any $(k_1,\ldots, k_n)\in S_0\smallsetminus\{(p_1,\ldots,p_n)\}$.
Then it holds that
\begin{equation*}
\begin{split}
x(k_1,k_2,&\ldots, k_n)=(-1)^{n-1}(x(k_1,p_2,p_3,\ldots, p_n)
  +x(p_1,k_2,p_3,\ldots, p_n)\\
& +\cdots+x(p_1,p_2,\ldots, p_{n-1},k_n) +(n-1) x(p_1,p_2,\ldots, p_n)).
\end{split}
\end{equation*}
\end{lemma}

\begin{proof}
We show the assertion by induction on $n$.
If $n=2$ then the assertion 
$$g(\eee_{k_1}\otimes\eee_{k_2})=-g(\eee_{k_1}\otimes\eee_{p_2}+\eee_{p_1}\otimes\eee_{k_2})-g(\eee_{p_1}\otimes\eee_{p_2})$$
%%$$x_{k_1k_2}=-(x_{k_1p_2}+x_{p_1k_2})-x_{p_1p_2}$$
follows from
$$(\eee_{k_1}+\eee_{p_1}) \otimes (\eee_{k_{2}}+\eee_{p_{2}}) = 
\eee_{k_1}\otimes\eee_{k_2}+(\eee_{k_1}\otimes\eee_{p_2}
\eee_{p_1}\otimes\eee_{k_2})+\eee_{p_1} \otimes \eee_{p_2}.
$$
Put 
$$W_n=\eee(k_1,p_2,\ldots,\eee_{p_{n}})+\eee(p_1,k_2,p_3,\ldots,\eee_{p_{n}})
+\cdots+\eee(p_1,\ldots,\eee_{p_{n-1}},\eee_{k_{n}})$$
for short.
We have
\begin{equation*}
\begin{split}
(W_n+&n\eee(p_1,\ldots,p_n))\otimes(\eee_{k_{n+1}}+\eee_{p_{n+1}}) \\
&=\sum_{h=1}^n (\eee(p_1,\ldots,p_{h-1},k_h,p_{h+1},\ldots,p_n,k_{n+1}) \\
& \hskip10mm
  +\eee(p_1,\ldots,p_n)\otimes(\eee_{k_{n+1}}+\eee_{p_{n+1}})) 
 +W_n\otimes\eee_{p_{n+1}} \\
& =0.
\end{split}
\end{equation*}
%%%by the induction assumption for $n=2$.
%%%%%%
As the induction assumption, we assume that
$$
g((\eee_{k_1}+\eee_{p_1}) \otimes \cdots \otimes 
  (\eee_{k_{n}}+\eee_{p_{n}})) = 0
$$
implies 
$$
g(\eee(k_1,\ldots, k_n))= (-1)^{n-1}g(W_n +(n-1)\eee(p_1,\ldots,p_n))
$$
for any $(k_1,\ldots,k_n)$ and any $(p_1,\ldots,p_n)$.
%%%%%%
Then we have
\begin{equation*}
\begin{split}
0&=g((\eee_{k_1}+\eee_{p_1}) \otimes \cdots \otimes (\eee_{k_{n}}+\eee_{p_{n}})\otimes (\eee_{k_{n+1}}+\eee_{p_{n+1}}))\\
%%%
& = g((\eee(k_1,\ldots,k_{n})+(-1)^n(W_n
  +(n-1)\eee(p_1,\ldots,p_{n}))) \otimes (\eee_{k_{n+1}}+\eee_{p_{n+1}}))
\\
%%%
& = g((\eee(k_1,\ldots,k_{n})-(-1)^n\eee(p_1,\ldots,p_{n}))\otimes (\eee_{k_{n+1}}+\eee_{p_{n+1}}))
\\
& = g(\eee(k_1,\ldots,k_{n+1}) +(-1)^{n-1}(W_n+(n-1)\eee(p_1,\ldots,p_{n}))
  \otimes\eee_{p_{n+1}} \\
& \hskip10mm 
  -(-1)^n\eee(p_1,\ldots,p_n)\otimes(\eee_{k_{n+1}}+\eee_{p_{n+1}})) \\
& = g(\eee(k_1,\ldots,k_{n+1}) 
  -(-1)^{n}[W_{n+1}+n\eee(p_1,\ldots,p_{n+1})]) \\
\end{split}
\end{equation*}
Therefore the assertion holds for $n+1$.
%%%%
\end{proof}

\begin{lemma} \label{lem:expand}
We suppose that $v_1=1$ if $n=1$.
If
$$g((\eee_{i'_1}+v_1\eee_{p_1})\cdots(e_{i'_n}+v_n\eee_{p_n}))=0$$
for any $1\leq i'_j\leq p_j$, $j=1,\ldots,n$ such that 
$(i'_1,\ldots,i'_n)\ne (p_1,\ldots,p_n)$
then
$$g((e_{k_1}+u_1\eee_{p_1})\cdots(e_{k_n}+v_n\eee_{p_n}))=
(u_1-1)\cdots(u_k-1)x(p_1,\ldots,p_n).$$
\end{lemma}

\begin{proof}
We show the assertion by induction on $n$.
If $n=1$ then 
\begin{equation*}
\begin{split}
g(e_{k_1}+u_1\eee_{p_1})&=
g((e_{k_1}+u_1\eee_{p_1})-(e_{k_1}+v_1\eee_{p_1})) \\
&=(u_1-1)x({p_1}).
\end{split}
\end{equation*} 
%%%%%
As the induction assumption, we assume that the assertion holds
for $n$ and any $p_1,\ldots,p_n$.
Putting $\beta=u_1(i_1,i_2,\ldots,k_{n+1})$, we have
\begin{equation*}
\begin{split}
g((e_{k_1}&+u_1\eee_{p_1})\otimes\cdots\otimes(e_{k_{n+1}}+v_{n+1}\eee_{p_{n+1}})) \\
& =g((e_{k_1}+u_1\eee_{p_1})\otimes\cdots\otimes(e_{k_{n}}+v_{n}\eee_{p_{n}})\otimes(e_{k_{n+1}}+\beta\eee_{p_{n+1}})) \\
& \hskip10mm 
  +(u_{n+1}-\beta)g((e_{k_1}+u_1\eee_{p_1})\otimes\cdots\otimes (e_{k_n}+u_n\eee_{p_n})\otimes\eee_{p_{n+1}})) \\
& =(u_1-1)\cdots(u_{n}-1)g(\eee(p_1,\ldots,p_n)\otimes(\eee_{k_{n+1}}+\beta\eee_{p_{n+1}}))) \\
& \hskip10mm 
  +(u_1-1)\cdots(u_n-1)(u_{n+1}-\beta)g(\eee(p_1,\ldots,p_{n})\otimes \eee_{p_{n+1}}) \\
& =(u_1-1)\cdots(u_{n}-1)g(\eee(p_1,\ldots,p_n)\otimes\eee_{k_{n+1}}) \\
& \hskip10mm 
  +(u_1-1)\cdots(u_{n}-1)u_{n+1}g(\eee(p_1,p_2,\ldots,p_{n+1})) \\
& =-1(u_1-1)\cdots(u_{n}-1)x(p_1,p_2,\ldots,p_{n+1}) \\
& +(u_1-1)\cdots(u_n-1)u_{n+1}x(p_1,p_2,\ldots,p_{n+1}) \\
& =(u_1-1)\cdots(u_{n+1}-1)x(p_1,p_2,\ldots,p_{n+1}). \\
%%%%
\end{split}
\end{equation*} 
We complete the proof.
\end{proof}

\begin{lemma} \label{lem:main:n=2}
Suppose that $n=2$, $2\leq p_1\leq p_2\leq p_3$,
$p_1p_2-p_1-p_2+3\leq p_3\leq p_1p_2$.
Then the equation $J(\varphi(\zzz))\xxx^T=\zerovec$ 
implies $\xxx=\zerovec$.
\end{lemma}

\begin{proof}
The equation $J(\varphi(\zzz))\xxx^T=\zerovec$ indicate
\begin{eqnarray}
x(i'_1,k_2,f(k_1,k_2))+u_2x(i'_1,p_2,f(k_1,k_2)) &=& 0, \label{eq:itauk:p0} \\
x(k_2,i'_2,f(k_1,k_2))+u_1x(p_1,i'_2,f(k_1,k_2)) &=& 0, \quad \label{eq:sigmajk:p0} \\
\begin{array}[b]{r} x(i_1,i_2,f(k_1,k_2))+v_1x(p_1,i_2,f(k_1,k_2))+v_2x(i_1,p_2,f(k_1,k_2)) \\
  +v_1v_2x(p_1,p_2,f(k_1,k_2)) \end{array} &=& 0, 
   \label{eq:ijk} 
\end{eqnarray}
for $1\leq i'_1\leq p_1$, $1\leq i'_2\leq p_2$, and $(i_1,i_2),(k_1,k_2)\in S$.
The equation \eqref{eq:ijk} for $(i'_1,i'_2)=(p_1,p_2)$ is
%%\eqref{eq:ijk} for $i'_1=p_1$, and
%%\eqref{eq:ijk} for $i'_2=p_2$ are
\begin{eqnarray} 
x(p_1,p_2,f(k_1,k_2))&=&0, \label{eq:mnk} 
%%x(p_1,k_2,f(k_1,k_2))&=&0, \label{eq:mjp0} \\
%%x(k_1,p_2,f(k_1,k_2))&=&0, \label{eq:inp0} 
\end{eqnarray}
%%%%%
Thus by \eqref{eq:mnk}, the equations \eqref{eq:itauk:p0} for $i'_1=p_1$ 
and \eqref{eq:sigmajk:p0} for $i'_2=p_2$ and \eqref{eq:ijk} are 
\begin{eqnarray}
x(p_1,k_2,f(k_1,k_2)) &=& 0 \label{eq:itauk:p0:new} \\
x(k_1,p_2,f(k_1,k_2)) &=& 0 \label{eq:sigmajk:p0:new} \\
x(i_1,i_2,f(k_1,k_2))+v_1x(p_1,i_2,f(k_1,k_2))+v_2x(i_1,p_2,f(k_1,k_2)) &=& 0 \label{eq:ijk:new}
\end{eqnarray}
%%for $1\leq i<p_1$, $1\leq j< n$ and $1\leq k< p_0$, 
for $1\leq i_1<p_1$, $1\leq i_2< p_2$ and $(i_1,i_2),(k_1,k_2)\in S$.
%%%%%%%
If $(k_1,k_2)=(p_1,p_2)$ then 
$$x(i_1,i_2,f(p_1,p_2))=0$$
for $1\leq i_1<p_1$ and $1\leq i_2< p_2$
by \eqref{eq:itauk:p0:new}, \eqref{eq:sigmajk:p0:new} and \eqref{eq:ijk:new}.
Put together with \eqref{eq:mnk}, \eqref{eq:itauk:p0:new} and \eqref{eq:sigmajk:p0:new},
we get
$$x(i'_1,i'_2,f(p_1,p_2))=0$$
for $1\leq i'_1\leq p_1$ and $1\leq i'_2\leq p_2$.
\par

Now we show that $x(i'_1,i'_2,f(k_1,k_2))=0$ for 
$1\leq i'_1\leq p_1$, $1\leq i'_2\leq p_2$, $(k_1,k_2)\in S$ and
$(k_1,k_2)\ne (p_1,p_2)$.
Suppose that $(k_1,k_2)\ne (p_1,p_2)$.
It follows from $(k_1,k_2)\in S$ that $k_1<p_1$ and $k_2<p_2$.
By combining \eqref{eq:itauk:p0} for $i'_1=i_1$, \eqref{eq:itauk:p0:new}
and 
\eqref{eq:ijk:new} for $i_2=k_2$,
we have
\begin{eqnarray*} 
(u_2(i_1,k_2)-u_2(k_1,k_2))x(i_1,p_2,f(k_1,k_2)) &=& 0
\end{eqnarray*}
for $1\leq i_1< p_1$.
Thus 
$$
x(i_1,p_2,f(k_1,k_2))=0
$$
for $1\leq i_1< p_1$, $i_1\ne k_1$.
Therefore $x(i'_1,p_2,f(k_1,k_2))=0$ for $1\leq i'_1\leq p_1$
by \eqref{eq:mnk} and \eqref{eq:sigmajk:p0:new}.
Similarly by combining 
\eqref{eq:sigmajk:p0} for $i'_2=i_2$, \eqref{eq:sigmajk:p0:new} and \eqref{eq:ijk:new}
for $i_1=k_1$,
we have
\begin{eqnarray*} 
(u_1(k_1,i_2)-u_1(k_1,k_2))x(p_1,i_2,f(k_1,k_2)) &=& 0
\end{eqnarray*}
which induces
$$x(p_1,i_2,f(k_1,k_2))=0$$ 
for $1\leq i_2< p_2$, $j\ne k_2$, 
and thus $x(p_1,i'_2,f(k_1,k_2))=0$ for $1\leq i'_2\leq p_2$ 
and $(k_1,k_2) \in S$
by \eqref{eq:mnk} and \eqref{eq:itauk:p0:new}.
Thus by \eqref{eq:ijk:new} again, we get
$x(i_1,i_2,f(k_1,k_2))=0$ for $1\leq i_1<p_1$, $1\leq i_2< p_2$.
Therefore $x(i'_1,i'_2,f(k_1,k_2))=0$ for 
$1\leq i'_1\leq p_1$, $1\leq i'_2\leq p_2$.
Consequently we get $\xxx=\zerovec$.
\end{proof}

%%%%%%%%%%%%%%%%%%%%%%%%%%%%%%%%%%%%%%%%%%%%
%%%%%%%%%%%%%%%%%%%%%%%%%%%%%%%%%%%%%%%%%%%%
\begin{thm} \label{thm:Jacobian}
The equation $J(\varphi(\zzz))\xxx^T=\zerovec$ 
implies $\xxx=\zerovec$ under the assumption in 
Theorem~\ref{thm:main}.
\end{thm}

\begin{proof}
%%%%
We consider the linear equation $J(\varphi(\zzz))\xxx^T=\zerovec$.
This equation is equivalent to
\begin{equation*} \label{eq:jacobian}
\psi_1(\aaa^{(1)}_k,\ldots,\aaa^{(n+1)}_k)\xxx^T=\zerovec,\ 1\leq k\leq p_n.
\end{equation*}
%%%%%
By \eqref{eq:phi1},
these equations 
%%\eqref{eq:jacobian} 
indicate the following:
%%%%%
\begin{eqnarray*}
g(\eee_{i'_1} \otimes \aaa^{(2)}_k \otimes \aaa^{(3)}_k\otimes \cdots \otimes \aaa^{(n+1)}_k) &=& 0, 
\label{eq:1} \\
g(\aaa^{(1)}_k \otimes \eee_{i'_2} \otimes \aaa^{(3)}_k\otimes \cdots \otimes \aaa^{(n+1)}_k) &=& 0, 
\label{eq:2} \\
\vdots \notag \\
g(\aaa^{(1)}_k \otimes \cdots \otimes \aaa^{(n-1)}_k \otimes \eee_{i'_n} \otimes \aaa^{(n+1)}_k) &=& 0, 
\label{eq:n} \\
g(\aaa^{(1)}_k \otimes \cdots \otimes \aaa^{(n-1)}_k \otimes \aaa^{(n)}_k \otimes \eee_{i'_{n+1}}) &=& 0. 
\label{eq:n+1} 
\end{eqnarray*}
for $1\leq k\leq p_n$. 
%%%%%%
In this proof, 
we always assume that $i'_j$ is taken over $1,2,\ldots,p_j$ 
for each $j=1,\ldots,n$.
Thus
%%%%%
\begin{eqnarray}
g((\eee_{i'_1} \otimes (\eee_{k_2}+u_2\eee_{p_2}) \otimes \cdots \otimes (\eee_{k_n}+u_n\eee_{p_n})\otimes \eee_{f(k_1,\ldots,k_{n})}) &=& 0, 
\label{eq:1:a} \\
g((\eee_{k_1}+u_1\eee_{p_1}) \otimes \eee_{i'_2} \otimes (\eee_{k_3}+u_3\eee_{p_3})\otimes \cdots \otimes \eee_{f(k_1,\ldots,k_{n})}) &=& 0, 
\label{eq:2:a} \\
\vdots\hspace{20mm} \notag \\
g((\eee_{k_1}+u_1\eee_{p_1}) \otimes \cdots \otimes (\eee_{k_{n-1}}+u_{n-1}\eee_{p_{n-1}}) \otimes \eee_{i'_n} \otimes \eee_{f(k_1,\ldots,k_{n})}) &=& 0, 
\label{eq:n:a} \\
g((\eee_{i_1}+v_1\eee_{p_1}) \otimes \cdots \otimes (\eee_{i_{n}}+v_{n}\eee_{p_{n}}) \otimes \eee_{f(k_1,\ldots,k_n)}) &=& 0. 
\label{eq:n+1:a} 
\end{eqnarray}
for any $(i_1,\ldots,i_n),(k_1,\ldots,k_n)\in S$. 
%%%%%%%%%%%%%%

We show the assertion by induction on $n$.
The assertion for $n=2$ holds by Lemma~\ref{lem:main:n=2}.
We suppose that $n\geq 3$ and 
the assertion holds for $n-1$ as the induction assumption.

By putting $(i'_1,\ldots,i'_n)=(p_1,\ldots,p_n)$, we get
\begin{equation}
x(p_1,\ldots,p_n,f(k_1,\ldots,k_n)) = 0 \label{eq:p:0}
\end{equation}
for any $(k_1,\ldots,k_n)\in S$.
Now let $k_n=p_n$.  Put $f_1=f(k_1,\ldots,k_{n-1},p_n)$ for short.
Then $u_1=\cdots=u_{n-1}=1$ and $u_n=0$.
By the $n$ equations \eqref{eq:1:a}-\eqref{eq:n:a},
the induction assumption yields us 
\begin{equation}
x(i'_1,\ldots,i'_{n-1},{p_n},f_1)=0 \label{eq:pn:0}
\end{equation}
for any $(k_1,\ldots,k_{n-1},p_n) \in S$ and any $i'_1,\ldots,i'_{n-1}$.
Then, by \eqref{eq:n+1:a} we get
\begin{equation} \label{eq:eq:n+1:b}
g((\eee_{i_1}+v_1\eee_{p_1}) \otimes \cdots 
\otimes (\eee_{i_{n-1}}+v_{n-1}\eee_{p_{n-1}}) \otimes 
  \eee_{i_n} \otimes \eee_{f_1}) = 0 
\end{equation}
for all $(i_1,\ldots,i_n)\in S$.
This equation and \eqref{eq:n:a} indicate
\begin{equation}
x(p_1,\ldots,p_{n-1},i_n,f_1)=0 \label{eq:pn:1}
\end{equation}
by Lemma~\ref{lem:expand} if $i_n<p_n$.
Suppose that $i_n<p_n$.
In the equation \eqref{eq:eq:n+1:b} we put $i_j=p_j$ for $n-2$ numbers $j$'s
with $j<n$ and get 
\begin{equation*}
x(i_1,p_2,\ldots,p_{n-1},i_n,f_1)=\cdots
=x(p_1,\ldots,p_{n-2},i_{n-1},i_{n},f_1)=0 %%\label{eq:pn:2}
\end{equation*}
for $1\leq i_j<p_j$, $j=1,\ldots,n$, and thus
\begin{equation}
x(i'_1,p_2,\ldots,p_{n-1},i_n,f_1)=\cdots
=x(p_1,\ldots,p_{n-2},i'_{n-1},i_{n},f_1)=0 \label{eq:pn:2}
\end{equation}
for any $i'_1,\ldots,i'_n$ by \eqref{eq:pn:1}.
In the equation \eqref{eq:eq:n+1:b} we put $i_j=p_j$ for $n-3$ numbers $j$'s
and get 
\begin{equation*}
x(i_1,i_2,p_3,\ldots,p_{n-1},i_n,f_1)=\cdots
=x(p_1,\ldots,p_{n-3},i_{n-2},i_{n-1},i_{n},f_1)=0 %%\label{eq:pn:2}
\end{equation*}
for $1\leq i_j<p_j$, $j=1,\ldots,n$, and thus
\begin{equation*}
x(i'_1,i'_2,p_3,\ldots,p_{n-1},i_n,f_1)=\cdots
=x(p_1,\ldots,p_{n-3},i'_{n-2},i'_{n-1},i_{n},f_1)=0 %%\label{eq:pn:3}
\end{equation*}
by \eqref{eq:pn:2}.
And go on, finally we get
\begin{equation*}
x(i'_1,\ldots,i'_{n-1},i_{n},f_1)=0 \label{eq:pn:(n-1)}
\end{equation*}
for any $i'_1,\ldots,i'_{n-1}$ and any $1\leq i_n<p_n$
and then by \eqref{eq:pn:0}
$$x(i'_1,\ldots,i'_{n-1},i'_{n},f_1)=0$$
for any $i'_1,\ldots,i'_n$.
%%%%%
If we consider the similar argument for $j$ instead of $n$,
we have 
$$x(i'_1,\ldots,i'_n,f(k_1,\ldots,k_n))=0$$
for any $i'_1,\ldots, i'_n$ and any $(k_1,\ldots,k_n)\in S$ with $k_j=p_j$
for some $j$.
\par
To complete the proof, it suffices to show that
$$x(i'_1,\ldots,i'_n,f(k_1,\ldots,k_n))=0$$
for any $i'_1,\ldots, i'_n$ and any $(k_1,\ldots,k_n)\in S$ with $k_j<p_j$
for each $j$.
Let $f_2=f(k_1,\ldots,k_n)$ for short.
By putting $i_n=p_n$ in \eqref{eq:n+1:a}, we get
$$g((\eee_{i_1}+\eee_{p_1}) \otimes\cdots\otimes (\eee_{i_{n-1}}+\eee_{p_{n-1}}) \otimes \eee_{p_n}\otimes \eee_{f_2}) = 0$$ 
for $(i_1,\ldots,i_{n-1},p_n)\in S$.
By Lemma~\ref{lem:codim1}, we have
\begin{equation*}
\begin{split}
%%(-1)^{n-2}
0 &=g((\eee(i_1,p_2,\ldots,p_{n-1})+\cdots+\eee(p_1,\ldots,p_{n-2},i_{n-1}) \\
& \hskip10mm
  +(n-2)\eee(p_1,\ldots,p_{n-1}))\otimes \eee(p_n,f_2)) \\
& =g((\eee(i_1,p_2,\ldots,p_{n-1})+\cdots+\eee(p_1,\ldots,p_{n-2},i_{n-1}))\otimes \eee(p_n,f_2)). 
\end{split}
\end{equation*}
Thus 
$$g((\eee(i_1,p_2,\ldots,p_{n})+\cdots+\eee(p_1,\ldots,p_{n-2},i_{n-1},p_n))\otimes \eee_{f_2})=0.$$
Similarly, for each $j=1,\ldots,n-1$, 
by putting $i_j=p_j$ in \eqref{eq:n+1:a}, we get
\begin{eqnarray*}
g((\eee(p_1,i_2,p_3,\ldots,p_{n})+\cdots+\eee(p_1,\ldots,p_{n-1},i_{n}))\otimes \eee_{f_2})&=&0, \\
\vdots \hskip20mm \notag \\
\begin{array}[b]{r}
  g((\eee(i_1,p_2,\ldots,p_{n})+\cdots+\eee(p_1,\ldots,p_{p-3},p_{i-2},p_{n-1},p_{n}) \\
  +\eee(p_1,\ldots,p_{n-1},i_{n}))\otimes \eee_{f_2})
\end{array} &=&0.
\end{eqnarray*}
Since 
$$\left|\begin{pmatrix} 1&\cdots&1\\ \vdots&&\vdots\\ 1&\cdots&1\end{pmatrix}
-E_n\right|=(-1)^{n-2}(n-1),$$
we have
$$x(i_1,p_2,\ldots,p_{n},f_2)=\cdots=x(p_1,\ldots,p_{n-1},i_n,f_2)=0$$
for $1\leq i_j<p_j$, $j=1,\ldots,n$, 
and then
$$x(i'_1,p_2,\ldots,p_{n},f_2)=\cdots=x(p_1,\ldots,p_{n-1},i'_n,f_2)=0$$
for all $i'_1,\ldots,i'_n$, 
since  $x(p_1,p_2,\ldots,p_{n},f_2)=0$.
By putting $i'_j=p_j$ for $n-2$ numbers $j$'s in the equation 
\eqref{eq:n+1:a} we get 
$$x(i_1,i_2,p_3,\ldots,p_n,f_2)=\cdots
=x(p_1,\ldots,p_{n-2},i_{n-1},i_{n},f_2)=0$$
for $1\leq i_j<p_j$, $j=1,\ldots,n$, 
and then
$$x(i'_1,i'_2,p_3,\ldots,p_n,f_2)=\cdots
=x(p_1,\ldots,p_{n-2},i'_{n-1},i'_{n},f_2)=0$$
for all $i'_1,\ldots,i'_n$. 
And so on, we finally get
$$x(i'_1,\ldots,i'_n,f_2)=0$$
for all $i'_1,\ldots,i'_n$.
We complete the proof.
\end{proof}

%%%%%%

Now we show Theorem~\ref{thm:main}.
\medskip

\begin{proof}[of Theorem~\ref{thm:main}]
Let $r$ be a typical rank of $p_1\times \cdots\times p_{n+1}$ tensors. 
Then $p_{n+1}\leq r\leq p_1p_2\cdots p_n$
by Lemma~\ref{lem:minimal typical rank}.
In particular, note that any integer less than $p_{n+1}$ is not a typical rank.
Since $p_{n+1}\geq q$, it holds that 
$p_{n+1}$ is a typical rank by Theorem~\ref{thm:Jacobian}.
\par
Conversely suppose that $p_{n+1}$ is a typical rank of 
$p_1\times \cdots\times p_{n+1}$ tensors.
By Proposition~\ref{prop:typical_rank_range}, 
$$p_{n+1}\geq \frac{p_1\cdots p_{n+1}}{p_1+\cdots+p_{n+1}-n}$$
which implies that
$p_{n+1}\geq q$,
and also, a typical rank is less than or equal to $p_1\cdots p_n$.
Thus $p_{n+1}\leq p_1\cdots p_n$.
We complete the proof.
\end{proof}

%\nocite{MR2377556,RePEc:spr:psycho:v:65:y:2000:i:4:p:525-532,Generic:2009,MR1693919,typical:2000}
%\bibliographystyle{amsplain}
%\bibliography{../../tensor}

\providecommand{\bysame}{\leavevmode\hbox to3em{\hrulefill}\thinspace}
\providecommand{\MR}{\relax\ifhmode\unskip\space\fi MR }
\providecommand{\MRhref}[2]{%
  \href{http://www.ams.org/mathscinet-getitem?mr=#1}{#2}
}
\providecommand{\href}[2]{#2}

\end{document}